\documentclass{amsart}
\usepackage{bm}
\usepackage{amscd, amssymb, amsmath}
\usepackage{amsthm}
\usepackage{graphicx}

\newtheorem{theorem}{Theorem}[section]
\newtheorem{lemma}[theorem]{Lemma}
\newtheorem{proposition}[theorem]{Proposition}
\newtheorem{corollary}[theorem]{Corollary}

\theoremstyle{definition}
\newtheorem{definition}[theorem]{Definition}

\numberwithin{equation}{section}

\newtheorem{remark}[theorem]{Remark}

\newcommand{\Pb}{\mathbb{P}}

\newcommand{\ch}{{\rm ch}}
\newcommand{\Par}{\mathcal{P}}
\newcommand{\Sg}{\mathbb{S}}

\title{On the cohomology of the Losev--Manin moduli space} 
\author{Jonas Bergstr\"om}
\address{Matematiska institutionen, Stockholms Universitet, 106 91 Stockholm, Sweden.}
\email{jonasb@math.su.se}

\author{Satoshi Minabe}
\address{Department of Mathematics, Tokyo Denki University, 120-8551 Tokyo,  Japan}
\email{minabe@mail.dendai.ac.jp}


\begin{document}
\begin{abstract}
We determine the cohomology of the Losev--Manin moduli space $\overline{M}_{0, 2 | n}$ of pointed genus zero curves as a representation of the product of symmetric groups $\Sg_2 \times \Sg_n$. 
\end{abstract}
\maketitle
\setlength{\baselineskip}{1.1\baselineskip}

\section*{Introduction}
The Losev--Manin moduli space $\overline{M}_{0, 2 | n}$ was introduced in \cite{LM} and it parametrizes stable chains of projective lines with marked points $x_0 \neq x_{\infty}$ and $y_1,\ldots,y_n$, where the points $y_1,\ldots,y_n$ are allowed to collide, but not with $x_0$ nor $x_{\infty}$, see Definition \ref{def:definition}. In \cite{LM} this moduli space was denoted by $\overline{L}_{n}$, here we have adapted the notation used in \cite{MOP}. There is a natural action of $\Sg_2 \times \Sg_n$  on $\overline{M}_{0, 2 | n}$ by permuting $x_0,x_{\infty}$ and $y_1,\ldots,y_n$ respectively. This makes the cohomology $H^*(\overline{M}_{0, 2 | n},\mathbb{Q})$ into a representation of $\Sg_2 \times \Sg_n$. The aim of this note is to determine the character of this representation. 

The moduli space $\overline{M}_{0, 2 | n}$ can also be described as a moduli space of weighted pointed curves which were studied by Hassett \cite[Section 6.4]{Hassett}. In this terminology it is the moduli space of genus~$0$ curves with $2$ points of weight $1$ and $n$ points of weight $1/n$, and it would be written $\overline{M}_{0,\mathcal{A}}$ where $\mathcal{A}=(1,1,\underbrace{1/n,\ldots,1/n}_n)$. 

Another interesting aspect of the space $\overline{M}_{0, 2 | n}$ is that it has a structure of toric variety. It is proved in \cite{LM} that $\overline{M}_{0, 2 | n}$ is isomorphic to the smooth projective toric variety associated with the convex polytope called the permutahedron. This toric variety is obtained by an iterated blow-up of $\Pb^{n-1}$ formed by first blowing up $n$ general points, then blowing up the strict transforms of the lines joining pairs among the original $n$ points, and so on up to $(n-3)$-dimensional hyperplanes, see \cite[\S 4.3]{Kap}. With this perspective, the action of $\Sg_2 \times \Sg_n$ can be seen in the following way. The $\Sg_n$-action comes from permuting the $n$-points of the blow-up, and the action of $\Sg_2$ comes from the Cremona transform of $\Pb^{n-1}$ induced by the group inversion of the torus $(\mathbb{C}^*)^{n-1}$ : $(t_1, \ldots, t_{n-1})\mapsto (t_1^{-1}, \ldots, t_{n-1}^{-1})$. 

Alternatively, we can view our moduli space $\overline{M}_{0, 2 | n}$ as the toric variety $X(A_{n-1})$ associated to the fan formed by Weyl chambers of the root system of type $A_{n-1}$ ($n\geq 2$), see \cite{Batyrev-Blume}. The cohomology of $X(A_{n-1})$ is a representation of the Weyl group $W(A_{n-1})\cong \Sg_n$ and this representation was studied in \cite{P,DV,Ste, Le}. On the other hand, $X(A_{n-1})$ has another automorphism coming from that of the Dynkin diagram. This automorphism together with the action of the Weyl group corresponds precisely to the $\Sg_2 \times \Sg_n$-action on $\overline{M}_{0, 2 | n}$. 

The cohomology of the moduli space $\overline{M}_{0, 2 | n}$ has also been studied by mathematical physicists, since it corresponds to the solutions of the so-called commutativity equations. For this perspective we refer to \cite{LM, SZ} and the references therein.

The outline of the paper is as follows. In Section~\ref{sec:moduli} we define $\overline{M}_{0, 2 | n}$ and we state some known results on its cohomology. Our main result is Theorem~\ref{thm:formula} where we give a formula for the $\Sg_2 \times \Sg_n$-equivariant Poincar\'e-Serre polynomial of $\overline{M}_{0, 2 | n}$.  The main theorem is formulated in Section~\ref{sec:result} and it is proved in Section~\ref{sec:proof}. In Section~\ref{sec:gener} we present a formula for the generating series of the $\Sg_2 \times \Sg_n$-equivariant Poincar\'e-Serre polynomial of $\overline{M}_{0, 2 | n}$. In Appendix~\ref{sec:procesi} we then show that the result of Procesi in \cite{P} on the $\Sg_n$-equivariant Poincar\'e-Serre polynomial is in agreement with our result. Finally in Appendix~\ref{sec:appendix} we list the  $\Sg_2 \times \Sg_n$-equivariant Poincar\'e-Serre polynomial of $\overline{M}_{0, 2 | n}$ for $n$ up to $6$.

\subsection*{Acknowledgement} The authors thank the Max--Planck--Institut f\"ur Mathematik for hospitality during the preparation of this note. 
The second named author is supported in part by JSPS Grant-in-Aid for Young Scientists (No.\,22840041).

\section{The moduli space $\overline{M}_{0, 2 | n}$} \label{sec:moduli}
In this note, a curve means a compact and connected curve over $\mathbb{C}$ with at most nodal singularities and the genus of a curve is the arithmetic genus.

\begin{definition}\label{def:definition}
For $n\geq 1$, let $\overline{M}_{0, 2 | n}$ be the moduli space of genus $0$ curves $C$ with $n+2$ marked points $(x_0, x_{\infty} | y_1, \ldots, y_{n})$ satisfying the following conditions:
\begin{itemize}
\item[(i) ] all the marked points are non-singular points of $C$,
\item[(ii)] $x_0$ and $x_{\infty}$ are distinct,
\item[(iii)] $y_1, \ldots, y_{n}$ are distinct from $x_0$ and  $x_{\infty}$,
\item[(iv)] the components corresponding to the ends of the dual graph contain $x_0$ or $x_{\infty}$,
\item[(v)] each component has at least three special (i.e. marked or singular) points.
\end{itemize} 
\end{definition}

\begin{remark}
In (iii) above, $y_i$ and $y_j$ are allowed to coincide. The conditions imply that the dual graph of $C$ is linear and that each irreducible component must contain at least one marked point in $(y_1, \ldots, y_n)$. This means that $C$ is a chain of projective lines of length at most $n$.
\end{remark}

The moduli space $\overline{M}_{0, 2 | n}$ is a nonsingular projective variety of dimension $n-1$, see \cite[Theorem 2.2]{LM}. It has an action of $\Sg_2 \times \Sg_n$ by permuting the marked points $(x_0,x_{\infty}|y_1, \ldots, y_n)$. 

\subsection{Cohomology of $\overline{M}_{0, 2 | n}$}
The cohomology ring $H^*(\overline{M}_{0, 2 | n},\mathbb{Q})$ was studied in \cite{LM}. It is algebraic, i.e., all the odd cohomology groups are zero and $H^*(\overline{M}_{0, 2 | n},\mathbb{Q})$ is isomorphic to the Chow ring $A^*(\overline{M}_{0, 2 | n},\mathbb{Q})$, see \cite[Theorem 2.7.1]{LM}. The Poincar\'{e}-Serre polynomials 
\begin{equation*}
E_{2|n}(q)=\sum_{i=0}^{n-1} 
\dim_{\mathbb{Q}} H^{2i} (\overline{M}_{0, 2 | n},\mathbb{Q}) \, q^i \in \mathbb{Z}[q]\, , 
\end{equation*}
were also computed, see \cite[Theorem 2.3]{LM}. 

The action of $\Sg_2 \times \Sg_n$ on $\overline{M}_{0, 2 | n}$ gives the cohomology $H^*(\overline{M}_{0, 2 | n},\mathbb{Q})$ a structure of $\Sg_2 \times \Sg_n$ representation. In \cite{P}, Procesi computed the $\Sg_n$-equivariant Poincar\'e-Serre polynomial of the toric variety $X(A_{n-1})$ (which is isomorphic to $\overline{M}_{0, 2 | n}$), see Appendix~\ref{sec:procesi}. 

Throughout this note the coefficients of all cohomology groups will be $\mathbb{Q}$.

\section{Statement of the result} \label{sec:result}
\subsection{Partitions}
A partition $\lambda=(\lambda_1 \geq \lambda_2 \geq \cdots)$ is a non-incresing sequence of non-negative integers which contains only finitely many non-zero $\lambda_i$'s.  The number $l(\lambda)$ of positive entries is called the \emph{length} of $\lambda$. The number $|\lambda|:=\sum_{i} \lambda_i$ is called the \emph{weight} of $\lambda$. If $|\lambda |=n$ we say that $\lambda$ is a partition of $n$. We denote by $\Par(n)$ the set of partitions of $n$ and by $\Par$ the set of all partitions. A sequence
$$ 
w\cdot \lambda
=(\lambda_{w(1)}, \lambda_{w(2)}, \ldots)\, ,\quad w \in \Sg_{l(\lambda)}~,
$$
obtained by permuting the non-zero elements of $\lambda$ is called an ordered partition of $n$. The number $c_{\lambda}$ of distinct ordered partitions obtained from 
$\lambda$ is given by 
$$
c_{\lambda}
=
\frac{l(\lambda)!}{\#{\rm Aut}(\lambda)}~,
$$ 
where ${\rm Aut}(\lambda)$ is the subgroup of $\Sg_{l(\lambda)}$ 
consisting of the permutations which preserve $\lambda$. 
Let $m_k(\lambda):=\#\{ i \mid \lambda_i =k\}$, we then have 
$$\#{\rm Aut}(\lambda)=
\prod_{k \geq 1} (m_k(\lambda)!)~.
$$
With this notation a partition $\lambda$ can also be written as $\lambda=[1^{m_1(\lambda)}\, 2^{m_2(\lambda)}\, \cdots]$. For $\lambda \in \Par(n)$ and $\mu \in \Par(m)$ we then define $\lambda+\mu \in \Par(m+n)$ by $m_k(\lambda+\mu):=\#\{ i \mid \lambda_i =k\}+\#\{ i \mid \mu_i =k\}$. 

\subsection{Symmetric functions}
For proofs of the statements in this section see for instance~\cite{Mac}. Let $\Lambda^y:=\underset{\longleftarrow}{\lim}~\mathbb{Z}[y_1, \ldots, y_n]^{\Sg_n}$ be the ring of symmetric functions. Similarly we define $\Lambda^{x|y}:=\Lambda^x\otimes \Lambda^y$. It is known that $\Lambda^y \otimes \mathbb{Q}=\mathbb{Q}[p_1^y, p_2^y, \ldots]$ where $p_n^y$ are the power sums in the variable $y$. For $\lambda\in \Par$, we set $p_{\lambda}^y:=\prod_i p_{\lambda_i}^y$. 

For a representation $V$ of $\Sg_n$, we define 
\begin{equation*}
{\rm ch}^y_n(V)
:=\frac{1}{n!} \sum_{w \in \Sg_n}
{\rm Tr}_V(w) p_{\rho(w)}^y \in \Lambda^y \, ,
\end{equation*}
where $\rho(w)\in \Par(n)$ is the partition of $n$ which represents the cycle type of $w\in \Sg_n$. Similarly we define, for a $\Sg_2 \times \Sg_n$ representation $V$, 
\begin{equation*}
{\rm ch}^{x|y}_{2|n}(V)
:=\frac{1}{2(n!)} \sum_{(v,w) \in \Sg_2 \times \Sg_n}
{\rm Tr}_V\bigl((v,w)\bigr) p_{\rho(v)}^x \, p_{\rho(w)}^y \in \Lambda^{x|y} \, . 
\end{equation*}

Recall that irreducible representations of $\Sg_n$ are indexed by $\Par(n)$. For $\lambda \in \Par(n)$, let $V_{\lambda}$ be the irreducible representation corresponding to $\lambda$ and define the Schur polynomial 
$$
s_{\lambda}^y:={\rm ch}^y_n(V_{\lambda})\in \Lambda^y \, .
$$

In the following we will use that, if $V_i$ are representations of $\Sg_{n_i}$ for $1 \leq i \leq k$, then  
\begin{gather*}
\ch^y_{\sum_{i=1}^k n_i} \Bigl({\rm Ind}_{\Sg_{n_1} \times \ldots \times \Sg_{n_k}}^{\Sg_{\sum_{i=1}^k n_i}} (V_1 \boxtimes \ldots \boxtimes V_k) \Bigr) = \prod_{i=1}^k \ch^y_{n_i}(V_i), \\
\ch^y_{n_1 n_2}\Bigl({\rm Ind}_{\Sg_{n_1} \sim \, \Sg_{n_2}}^{\Sg_{n_1 n_2}} (V_1 \boxtimes \underbrace{ V_2 \boxtimes \ldots \boxtimes V_2 }_{n_1} \,) \Bigr) = \ch^y_{n_1}(V_1) \circ \ch^y_{n_2}(V_2),
\end{gather*}
where $\sim$ denotes the wreath product, that is,  $\Sg_{n_1} \sim\,  \Sg_{n_2}:=\Sg_{n_1} \ltimes (\Sg_{n_2})^{n_1}$ where $\Sg_{n_1}$ acts on $(\Sg_{n_2})^{n_1}$ by permutation, see \cite[Appendix A, p. 158]{Mac}. Plethysm is an operation $\circ: \Lambda^y \times \Lambda^y \rightarrow \Lambda^y$ which we will extend to an operation $\circ: \Lambda^y \times \Lambda^y[q] \rightarrow \Lambda^y[q]$ by putting $p_n^y \circ q=q^n$.

\subsection{The main theorem}
\begin{definition}
The $\Sg_2 \times \Sg_n$-equivariant Poincar\'{e}-Serre polynomial of 
$\overline{M}_{0, 2 | n}$ is defined by 
\begin{equation*}
E_{\Sg_2 \times \Sg_n}(q):=\sum_{i=0}^{n-1}
{\rm ch}^{x|y}_{2|n}\left(H^{2i}(\overline{M}_{0, 2 | n})\right) q^{i}\in \Lambda^{x|y} [q]~.
\end{equation*}
\end{definition}

The usual Poincar\'e-Serre polynomial $E_{2|n}(q)$ is recovered from the equivariant one by 
$$
\frac{\partial^2}{\partial (p_1^x)^2}\frac{\partial^n}{\partial (p_1^y)^n} \, E_{\Sg_2 \times \Sg_n}(q) 
= E_{2|n}(q)\, .
$$

We will make some ad-hoc definitions in order to formulate an explicit formula for $E_{\Sg_2 \times \Sg_n}(q)$. The proof will then furnish an explanation to these definitions. 

\begin{definition}
First put $g_0^y:=1$, then for any $n \geq 1$ and any (unordered) partition $\lambda$ put  
\begin{equation*} 
f^{y}_{n}:=\sum_{i=0}^{n-1} (-1)^{i} s_{(n-i, 1^{i})}^y q^{n-1-i}, \quad F_{\lambda}^y:=\prod_{j=1}^{l(\lambda)} f^{y}_{\lambda_ j},\quad  g^{y}_{n}:=\sum_{i=0}^{n-1} s_{(n-i, 1^{i})}^y q^{n-1-i}. 
\end{equation*}
\end{definition} 

\begin{theorem}\label{thm:formula}
We then have 
\begin{equation}\label{eq:formula}
E_{\Sg_2 \times \Sg_n}(q)=
\frac{1}{2} (p_{1}^x)^2  \sum_{\lambda\in \Par(n) } c_{\lambda} \, F_{\lambda}^y + \frac{1}{2} p_{2}^x \sum_{k=0}^{\lfloor n/2 \rfloor} g^{y}_{n-2k} \sum_{\mu \in \Par(k)} c_{\mu} \bigl(p_{2}^y \circ F_{\mu}^y \bigr)~.
\end{equation} 
\end{theorem} 

Results for $1\leq n \leq 6$ obtained from \eqref{eq:formula} are listed in Appendix \ref{sec:appendix}.

\section{Proof of Theorem \ref{thm:formula}} \label{sec:proof} 
\subsection{Stratification of $\overline{M}_{0, 2 | n}$}
For $k\geq 0$, we denote by $\Delta_{n,k}$ the closed subset of $\overline{M}_{0, 2 | n}$ consisting of curves with at least $k$ nodes. Let $\Delta_{n, k}^{*}:=\Delta_{n, k}\setminus \Delta_{n, k+1}$ be the open part of $\Delta_{n, k}$ which corresponds to curves with exactly $k$ nodes. It is easy to see that $\Delta_{n,k}\neq \emptyset$ only for $0 \leq k \leq n-1$ and that $\Delta_{n, n-1}^*=\Delta_{n, n-1}=\{{\rm pt}\}$. Note that $\Delta_{n, k}^*$ is preserved by the $\Sg_2 \times \Sg_n$-action. Hence its cohomology $H^*(\Delta_{n,k}^*)$ is a representation of $\Sg_2 \times \Sg_n$. 

\begin{definition} For an ordered partition $\lambda$ of $n$ with length $k+1$, let $\Delta^*_{\lambda} \subset \Delta^*_{n,k}$ correspond to all chains of projective lines of length $k+1$ such that precisely $\lambda_i$ of the marked points $(y_1, \ldots, y_n)$ are on the $i$th component (where the component with the marked point $x_0$ is the $1$st component and the one with $x_{\infty}$ is the $(k+1)$th).
\end{definition}

Note that $\Delta^*_{\lambda}$ is preserved by $\Sg_n$ (but not necessarily by $\Sg_2 \times \Sg_n$, see below) and hence $H^*(\Delta_{\lambda}^*)$ is a representation of $\Sg_n$. 

\begin{lemma}\label{lem:strata}
(i) $\Delta_{n, 0}^* \cong \left( \mathbb{C}^* \right)^{n-1}$. 
(ii) $\Delta^*_{\lambda} \cong \prod_{i=1}^{l(\lambda)} \Delta_{\lambda_i, 0}^*$~. \\
(iii) We have a stratification  
\begin{equation*}
\Delta_{n, k}^*= \bigsqcup_{\lambda=(\lambda_1, \ldots, \lambda_{k+1})}
 \Delta^*_{\lambda}~~,
\end{equation*}
where $\lambda$ runs over all ordered partitions of $n$ with length $k+1$.  
\end{lemma}
\begin{proof}
(i) We have $\Delta_{n, 0}^*\cong  \left(\Pb^1\setminus\{0,\infty\} \right)^n / \mathbb{C}^* \cong  \left( \mathbb{C}^* \right)^n / \mathbb{C}^*$. (ii) Clear from the definition. (iii) This is found by considering the ways to distribute $n$ marked points $(y_1, \ldots, y_n)$ on the chain of projective lines of length $k+1$ so that each irreducible component contains at least one of the points. 
\end{proof}

It follows from Lemma \ref{lem:strata} (ii) that $\Delta^*_{\lambda}$ and $\Delta^*_{\lambda'}$ are ($\Sg_n$-equivariantly) isomorphic when $\lambda$ and $\lambda'$ are different orderings of the same element in $\Par(n)$.

\subsection{Cohomology of $\Delta_{n, 0}^*$}
Since $\Delta_{n, 0}^* \cong \left(\mathbb{C}^*\right)^{n-1}$, $H^i(\Delta_{n, 0}^*)=0$ for $i \geq n$, and moreover the mixed Hodge structure on $H_c^{2(n-1)-i} (\Delta_{n, 0}^*)$ is a pure Tate structure of weight $2(n-1-i)$, that is, $$H_c^{2(n-1)-i} (\Delta_{n, 0}^*)= \mathbb{Q} \bigl(-(n-1-i) \bigr)^{\oplus \binom{n-1}{i}}~~.$$
\begin{lemma}
For $0\leq i \leq n-1$,  we have
$$
{\rm ch}^{x|y}_{2|n} \left(H^i(\Delta_{n, 0}^*)\right) = 
\begin{cases} s_{(2)}^x \, s_{(n-i, 1^i)}^y \quad \text{if $i$ is even} \\ s_{(1^2)}^x \,  s_{(n-i, 1^i)}^y \quad \text{if $i$ is odd}. \end{cases}
$$
\end{lemma}
\begin{proof}
Take an isomorphism $\Delta_{n, 0}^*= \left(\mathbb{C}^*\right)^n / \mathbb{C}^* \to \left(\mathbb{C}^*\right)^{n-1}$ given by 
$$
(z_1: z_2:  \cdots : z_{n-1}: z_n) \mapsto 
(\frac{z_1}{z_n}, \ldots , \frac{z_{n-1}}{z_n})=:(y_1, \ldots, y_{n-1})~~.
$$
Then it is easy to see that 
$H^1(\Delta_{n, 0}^*)=\oplus_{i=1}^{n-1} \mathbb{Q} [\frac{1}{2\pi \sqrt{-1}}\frac{dy_i}{y_i}]$
is the standard representation $s_{(n-1, 1)}$ under the action of $\Sg_n$. The action of $\Sg_2$ is by interchanging $0$ and $\infty$, that is by the isomorphism $t \mapsto 1/t$ of $\Pb^1$, which induces the action $(z_1:  \cdots : z_n) \mapsto (1/z_1:  \cdots : 1/z_n)$ on $\Delta_{n, 0}^*$. This tells us that $(y_1, \ldots, y_{n-1}) \mapsto (1/y_1, \ldots, 1/y_{n-1})$ and since $\frac{d(1/y)}{1/y}=-\frac{dy}{y}$ we conclude that  $H^1(\Delta_{n, 0}^*) = V_{(1^2)} \boxtimes V_{(n-1, 1)}$. Using once more that $\Delta_{n, 0}^* \cong \left(\mathbb{C}^*\right)^{n-1}$ we get 
$$
H^k(\Delta_{n, 0}^*)\cong \wedge^k H^1(\Delta_{n, 0}^*) \cong 
\wedge^k ( V_{(1^2)} \boxtimes V_{(n-1, 1)} ) \cong (\otimes^k  V_{(1^2)} ) \boxtimes V_{(n-k, 1^k)}~~.
$$ 
\end{proof}

\begin{corollary}\label{cor:hodge}
We have the equality
\begin{equation*} 
\sum_{i=0}^{n-1} (-1)^i \, {\rm ch}^{x|y}_{2|n} \left(H_c^{2(n-1)-i} (\Delta_{n, 0}^*)\right) q^{n-1-i} =\frac{1}{2}(p_{1}^x)^2\,f^{y}_{n}+\frac{1}{2} p_{2}^x\,g^{y}_{n}~.
\end{equation*}
\end{corollary}
\begin{proof}
By Poincar\'{e} duality, $H_c^{2(n-1)-i} (\Delta_{n, 0}^*) \cong H^{i} (\Delta_{n, 0}^*)^{\vee}$, and since every irreducible representation of $\Sg_2 \times \Sg_n$ is defined over $\mathbb{Q}$, the dual representation  is isomorphic to itself. The equality now follows from the lemma together with the relations $2s_{(2)}^x=(p_1^x)^2+p_2^x$ and $2s_{(1^2)}^x=(p_1^x)^2-p_2^x$. 
\end{proof}

\subsection{Cohomology of $\Delta_{\lambda}^*$}
\begin{corollary}\label{cor:pure}
For any ordered partition $\lambda$ of $n$ with length $k+1$, $H_c^{2(n-k-1)-i} (\Delta_{\lambda}^*)$ is a pure Hodge structure of weight $2(n-k-1-i)$.
\end{corollary}
\begin{proof} This follows from Lemma \ref{lem:strata} (ii) and the purity of the cohomology of $\Delta_{i,0}^*$.
\end{proof}

\begin{corollary}\label{cor:lambda}
For any ordered partition $\lambda$ of $n$ with length $k+1$ we have 
\begin{equation*} 
\sum_{i=0}^{n-k-1} (-1)^i \, {\rm ch}^y_{n} \left(
H_c^{2(n-k-1)-i} (\Delta_{\lambda}^*)\right) q^{n-k-1-i} = F_{\lambda}^y~. 
\end{equation*}
\end{corollary}
\begin{proof} 
From Lemma~\ref{lem:strata} (ii) we know that $\Delta_{\lambda}^* \cong \prod_{i=1}^{k+1} \Delta_{\lambda_i, 0}^*$, and on each $\Delta_{\lambda_i,0}^*$ we have an action of $\Sg_{\lambda_i}$.  The action of $\Sg_n$ on $H^*_c(\Delta_{\lambda}^*)$ will thus be the induced action from $\Sg_{\lambda_1} \times \ldots \times \Sg_{\lambda_{k+1}}$ to $\Sg_n$. The result now follows from Corollary~\ref{cor:hodge}, forgetting the action of $\Sg_2$.
\end{proof}

\subsection{Proof of Theorem \ref{thm:formula}} 
We have the following long exact sequence of cohomology with compact support:
\begin{equation}\label{eq:exact}
\cdots 
\longrightarrow 
H^{i-1}_c (\Delta_{n, k+1}) 
\longrightarrow 
H^{i}_c (\Delta_{n, k}^*)
\longrightarrow
H^{i}_c (\Delta_{n, k})
\longrightarrow 
H^{i}_c (\Delta_{n, k+1}) 
\longrightarrow 
\cdots~~. 
\end{equation}
This is an exact sequence of both mixed Hodge structures and $\Sg_2 \times \Sg_n$-representations. Therefore, using the exact sequence \eqref{eq:exact} inductively (this is just the additivity of the Poincar\'e-Serre polynomial) we get 
\begin{equation} \label{eq:E}
E_{\Sg_2 \times \Sg_n} (q)
=
\sum_{k=0}^{n-1} 
\left\{
\sum_{i=0}^{n-1} (-1)^i \, 
{\rm ch}^{x|y}_{2|n}\left( H_c^{2(n-1)-i}(\Delta_{n, k}^*)\right)  q^{n-1-i}
\right\}~.
\end{equation}

We will now find a formula for ${\rm ch}^{x|y}_{2|n}\bigl( H_c^{2(n-1)-i}(\Delta_{n, k}^*)\bigr)$. Let us begin with a strata $\Delta^*_{\lambda}$ for an ordered partition $\lambda$ of $n$ with length $k+1$. The action of $\Sg_2$ will then send the strata given by $\lambda$ to the one given by $\lambda'=(\lambda_{k+1},\lambda_k,\ldots,\lambda_1)$. We will therefore divide into two cases. 

Let us first assume that $\lambda \neq \lambda'$. Since the action of $\Sg_2$ interchanges the two components it will also interchange the factors of $H_c^{i}( \Delta^*_{\lambda} \sqcup \Delta^*_{\lambda'})= H_c^{i}( \Delta^*_{\lambda} ) \oplus H_c^{i}( \Delta^*_{\lambda'} )$ and hence 
\begin{equation} \label{eq:case1}
{\rm ch}^{x|y}_{2|n}\bigl(H_c^{i}( \Delta^*_{\lambda} \sqcup \Delta^*_{\lambda'}) \bigr) = (p_{1}^x)^2  \,  {\rm ch}_{n}^y\bigl(H_c^{i}( \Delta^*_{\lambda} )\bigr)~. 
\end{equation}

Let us now assume that $\lambda = \lambda'$. We can then decompose our space as $\Delta^*_{\lambda}=\Delta^*_1 \times \Delta^*_2 \times \Delta^*_3$ where, if $k+1=2m$,  
$$\Delta^*_{1}:=\prod_{i=1}^m \Delta^*_{\lambda_i,0}, \quad \Delta^*_{2}:=\{{\rm pt}\}, \quad \Delta^*_{3}:=\prod_{i=m+1}^{2m} \Delta^*_{\lambda_i,0}~,$$ 
and, if $k+1=2m+1$,
$$\Delta^*_{1}:=\prod_{i=1}^m \Delta^*_{\lambda_i,0}, \quad \Delta^*_{2}:=\Delta^*_{\lambda_{m+1},0}, \quad \Delta^*_{3}:=\prod_{i=m+2}^{2m+1} \Delta^*_{\lambda_i,0}~.$$
Let us put $\alpha:=\lambda_{m+1}$ if $k+1$ is odd and $\alpha:=1$ if $k+1$ is even, and in both cases $\beta:=\sum_{i=1}^m \lambda_i$. The action of $\Sg_2$ interchanges the ($\Sg_{\beta}$-equivariantly) isomorphic components $\Delta_1^*$ and $\Delta_3^*$ and sends the space $\Delta^*_2$ to itself. Define the semidirect product $\Sg_2 \ltimes (\Sg_{\beta} \times \Sg_{\alpha} \times \Sg_{\beta})$ where $\Sg_2$ acts as the identity on $\Sg_{\alpha}$ and permutes the factors $\Sg_{\beta} \times \Sg_{\beta}$ (i.e. as the wreath product). The group $\Sg_2 \ltimes (\Sg_{\beta} \times \Sg_{\alpha} \times \Sg_{\beta})$ naturally embeds, by the map $i$ say, in $\Sg_{2\beta+\alpha}=\Sg_n$. Let us then put $\Sg_2 \ltimes (\Sg_{\beta} \times \Sg_{\alpha} \times \Sg_{\beta})$ in $\Sg_2 \times \Sg_{n}$ by $(\tau,\sigma) \mapsto (\tau,i(\tau,\sigma)$), where $\tau \in \Sg_2$ and $\sigma \in \Sg_{\beta} \times \Sg_{\alpha} \times \Sg_{\beta}$. The action of $\Sg_2 \times \Sg_n$ on $\Delta^*_{\lambda}$ will then be the induced action from $\Sg_2 \ltimes (\Sg_{\beta} \times \Sg_{\alpha} \times \Sg_{\beta})$ acting naturally on $\Delta_1^* \times \Delta_2^* \times \Delta_3^*$. Using Corollary~\ref{cor:hodge} we conclude that 
\begin{multline} \label{eq:case2} 
{\rm ch}^{x|y}_{2|n}\bigl(H_c^{i}( \Delta^*_{\lambda} ) \bigr) =\frac{1}{2} p_{(1^2)}^x \, f^{y}_{\alpha} \, \Bigl(p_{(1^2)}^y \circ \ch_{\beta}^y \bigl(H_c^{i}(\Delta^*_{1})\bigr) \Bigr) + \frac{1}{2} p_{(2)}^x \, g^{y}_{\alpha} \, \Bigl(p_{(2)}^y \circ \ch_{\beta}^y \bigl(H_c^{i}(\Delta^*_{1})\bigr) \Bigr) ~.
\end{multline}
Applying formula \eqref{eq:case1} and formula \eqref{eq:case2} (and using Lemma \ref{lem:strata} (iii) and Corollary \ref{cor:lambda}) to equation \eqref{eq:E}, gives equation \eqref{eq:formula}.

\section{Generating series} \label{sec:gener}
\subsection{Generating series of $E_{\Sg_2 \times \Sg_n} (q)$}
For any sequence of polynomials $h_n$ we have the formal identity, 
\begin{equation} \label{eq:geneq}
 1+\sum_{n=1}^{\infty} \Bigl( \sum_{\lambda \in \mathcal{P}(n)} c_{\lambda} \prod_{j=1}^{l(\lambda)} h_{\lambda_ j} \Bigr) =  1+ \sum_{r=1}^{\infty} \Bigl(\sum_{n=1}^{\infty} h_{n} \Bigr)^{r} =  \Bigl(1-\sum_{n=1}^{\infty} h_{n} \Bigr)^{-1}.
\end{equation}

The following proposition follows directly from \eqref{eq:geneq} and Theorem~\ref{thm:formula}.

\begin{proposition} \label{prop:gen2} The generating series of $E_{\Sg_2 \times \Sg_n} (q)$ is determined by, 
  \begin{equation}  \label{eq:gen2} 
1+\sum_{n=1}^{\infty} E_{\Sg_2\times \Sg_n}(q) = \frac{1}{2}(p_1^x)^2 \, \Bigl(1-\sum_{n=1}^{\infty} f_n^y\Bigr)^{-1}+\frac{1}{2} p_2^x \, \Bigl(1+\sum_{n=1}^{\infty} g_n^y \Bigr) \Bigl(1-\sum_{n=1}^{\infty} (p_2^y \circ f_n^y) \Bigr)^{-1}.    
  \end{equation}
\end{proposition}

\begin{remark} Consider the moduli space $M$ defined as in Definition~\ref{def:definition} but with the additional demand that $y_1,\ldots,y_n$ are distinct from each other. From Carel Faber we learnt the following formula, which is very similar to \eqref{eq:gen2}, for the generating series of the $\Sg_2 \times \Sg_n$-equivariant Poincar\'{e}-Serre polynomial of $M$. Carel Faber obtained the formula as a direct consequence of an equality he learned from Ezra Getzler. These results have not been published.

Let $h^y_{n+2}$ be the $\Sg_{n+2}$-equivariant Poincar\'{e}-Serre polynomial of $M_{0,n+2}$, the moduli space of genus~$0$ curves with $n+2$ marked distinct points. The $\Sg_2 \times \Sg_n$-equivariant Poincar\'{e}-Serre polynomial of the open part of $M$ (defined using the compactly supported Euler-characteristic) consisting of irreducible curves will then equal
$$\frac{1}{2} (p_1^x)^2 \, \tilde f^y_n+ \frac{1}{2} p_2^x \, \tilde g^y_n = \frac{1}{2} (p_1^x)^2 \, \Bigl( \frac{\partial^2 \, h^y_{n+2}}{\partial (p_1^y)^2}  \Bigr)+ \frac{1}{2} p_2^x \, \Bigl(2 \, \frac{\partial \, h^y_{n+2}}{\partial p_2} \Bigr).$$
From the proof of Theorem~\ref{thm:formula} we see that replacing $f^y_n$ by $\tilde f^y_n$ (and $g^y_n$ by $\tilde g^y_n$) in equation~\eqref{eq:gen2} gives the $\Sg_2 \times \Sg_n$-equivariant Poincar\'{e}-Serre polynomial of $M$.
\end{remark}

\begin{remark}
The polynomials $f_n^y$ and $g_n^y$ can be formulated in terms of $P^y_{\lambda}(q) \in \Lambda^y[q]$, the Hall--Littlewood symmetric function associated to $\lambda\in \Par$ (cf. \cite[III-2]{Mac}). This function is defined as the limit of the following symmetric polynomial: 
$$
P_{\lambda}(y_1, \ldots, y_k; q)
=
\sum_{w\in \Sg_k / \Sg_k^{\lambda}} 
w\left(
y_1^{\lambda_1} \cdots y_k^{\lambda_k}
\prod_{\lambda_i > \lambda_j} 
\frac{y_i - q y_j}{y_i - y_j} 
\right)\, , 
$$ 
where $\Sg_k^{\lambda}$ is the stabilizer subgroup of $\lambda$ in $\Sg_k$ and $l(\lambda)\leq k$ is assumed. In the special case $\lambda=(n)$, where $n \geq1$, the following formula is known (cf. \cite[p. 214]{Mac}): 
\begin{equation} \label{eq:hook}
P^{y}_{(n)}(q) =\sum_{r=0}^{n-1} (-q)^r s^y_{(n-r, 1^r)}~,  
\end{equation}
hence $f_n^y=q^{n-1}P^{y}_{(n)}(q^{-1})$ and $g_n^y=q^{n-1}P^{y}_{(n)}(-q^{-1})$.
\end{remark}

\subsection{Generating series of $E_{\Sg_n} (q)$}
 The $\Sg_n$-equivariant Poincar\'e-Serre polynomial of $\overline{M}_{0, 2 | n}$ equals 
\begin{equation*}
E_{\Sg_n}(q):=\sum_{i=0}^{n-1}
{\rm ch}^{y}_{n}\left(H^{2i}(\overline{M}_{0, 2 | n})\right) q^{i}
= \frac{\partial^2}{\partial (p_1^x)^2} E_{\Sg_2\times \Sg_n}(q)
\in \Lambda^y[q]~, 
\end{equation*}
and so 
\begin{equation} \label{eq:inverse}
1+\sum_{n=1}^{\infty} E_{\Sg_n}(q)=\Bigl(1-\sum_{n=1}^{\infty} f_n^y \Bigr)^{-1}.
\end{equation}
Corollary~\ref{cor:hodge} then tells us that the generating series of $E_{\Sg_n}(q)$ is the multiplicative inverse of the generating series (in compactly supported cohomology) of $\Delta^*_{n,0}$, which is the open part of $\overline{M}_{0, 2 | n}$ consisting of irreducible curves. 

If we set $q=1$, the Hall--Littlewood function $P^{y}_{(n)}(q^{-1})$ becomes the $n$th power sum $p_n^y$ and formula \eqref{eq:inverse} takes a very simple form. Let $e_{\Sg_n}:=E_{\Sg_n}(1) \in \Lambda^y$, be the $\Sg_n$-equivariant Euler characteristic of $\overline{M}_{0, 2 | n}$. We then have 
\begin{equation*}
1+ \sum_{n=1}^{\infty} e_{\Sg_n} z^n  
=\left(
1- \sum_{n=1}^{\infty} p_{n}^y z^n
\right)^{-1}\, .
\end{equation*}

\appendix
\section{Consistency with Procesi's result} \label{sec:procesi}
\subsection{Procesi's recursive formula}
In \cite{P}, Procesi obtained the following recursive relation among $E_{\Sg_{n}}(q)$ with respect to $n$. 
\begin{theorem}[Procesi]
The $E_{\Sg_n}(q)$ satisfy
\begin{equation*} 
E_{\Sg_{n+1}}(q)
=
s^y_{(n+1)} \sum_{i=0}^n q^i
+ \sum_{i=0}^{n-2}
s^y_{(n-i)}\, E_{\Sg_{i+1}}(q)\, 
\left(
\sum_{k=1}^{n-i-1} q^k
\right)~.
\end{equation*}
\end{theorem}
As a corollary, we have the following formula which is obtained in \cite{DV, Sta, Ste}.
\begin{corollary}\label{cor:ssdv}
We have
\begin{equation*} 
1+\sum_{n=1}^{\infty}E_{\Sg_n}(q) t^n
=\frac{(1-q)H(t)}{H(qt)-qH(t)}~,
\end{equation*}
where $H(t)=\sum_{r \geq 1} h_r t^r$ is the generating function of the complete symmetric functions in the variable $y$.
\end{corollary}

\subsection{Equivalence}
The following proposition shows the equivalence between our result and Procesi's by comparing Equation~\eqref{eq:inverse} and Equation~\eqref{eq:hook} to Corollary \ref{cor:ssdv}. 
\begin{proposition}
We have
\begin{equation*}
\frac{(1-q)H(t)}{H(qt)-qH(t)}
=\left\{
1-\sum_{r=1}^{\infty} q^{-1} P^{y}_{(r)} (q^{-1}) (qt)^r
\right\}^{-1}~.
\end{equation*}
\end{proposition}
\begin{proof} 
As in \cite[pp. 209--210]{Mac}, we have
\begin{multline*}
\frac{H(qt)}{H(t)}
=\prod_{i \geq 1} \frac{1-ty_i}{1-qty_i}  =1+ (1-q^{-1})\sum_{i=1}^n \frac{y_i qt}{1-y_i qt}
\prod_{j: j\neq i} \frac{y_i-q^{-1}y_j}{y_i-y_j}=\\
= 1+(1-q^{-1})\sum_{r=1}^{\infty}
P^{y}_{(r)}(q^{-1}) (qt)^r~.
\end{multline*}
An easy manipulation of this formula gives the wanted equality. \end{proof}

\section{$E_{\Sg_2 \times \Sg_n}(q)$ for $n$ up to $6$} \label{sec:appendix}
$$
\begin{tabular}{c|l}
\hline
$n$& $E_{\Sg_2 \times \Sg_n}(q)$\\
\hline
$1$& $s_{(2)}^x s_{(1)}^y$ \\
\hline
$2$& $(q+1) s_{(2)}^x s_{(2)}^y$ \\
\hline
$3$& $s_{(2)}^x \Bigl((q^2+q+1) s_{(3)}^y+ q \, s_{(2,1)}^y\Bigr)+q \, s^x_{(1^2)} s_{(3)}^y$\\
\hline
$4$ &$s_{(2)}^x \Bigl((q^3+2 q^2+2 q+1) s_{(4)}^y+(q^2+q) s_{(3,1)}^y+(q^2+q) s_{(2^2)}^y\Bigr)$\\
&$+s^x_{(1^2)} \Bigl((q^2+q) s_{(4)}^y+(q^2+q) s_{(3,1)}^y\Bigr)$\\
\hline
$5$ & $s_{(2)}^x \Bigl((q^4+2 q^3+4 q^2+2 q+1) s_{(5)}^y+(2 q^3+3 q^2+2 q) s_{(4,1)}^y$\\
&\qquad $+(q^3+3 q^2+q) s_{(3,2)}^y+q^2 \, s_{(2^2,1)}^y \Bigr)$\\
& $+s^x_{(1^2)} \Bigl((2 q^3+2 q^2+2 q)s_{(5)}^y+(q^3+3 q^2+q) s_{(4,1)}^y+(q^3+2 q^2+q) s_{(3,2)}^y+q^2 \, s_{(3,1^2)}^y \Bigr)$\\
\hline
$6$ &$s_{(2)}^x \Bigl((q^5+3 q^4+6 q^3+6 q^2+3 q+1) s_{(6)}^y+(2 q^4+6 q^3+6 q^2+2 q) s_{(5,1)}^y$\\
&\qquad$+(2 q^4+7 q^3+7 q^2+2 q) s_{(4,2)}^y+(q^3+q^2) s_{(4,1^2)}^y+(2 q^3+2 q^2) s_{(3^2)}^y$\\
&\qquad$+(2 q^3+2 q^2) s_{(3,2,1)}^y+(q^3+q^2) s_{(2^3)}^y\Bigr)$ \\
&$+s^x_{(1^2)}  \Bigl((2 q^4+4 q^3+4 q^2+2 q) s_{(6)}^y+(2 q^4+6 q^3+6 q^2+2 q) s_{(5,1)}^y$\\
&\qquad$+(q^4+5 q^3+5 q^2+q) s_{(4,2)}^y+(2 q^3+2 q^2) s_{(4,1^2)}^y$\\
&\qquad$+(q^4+3 q^3+3 q^2+q) s_{(3^2)}^y+(2 q^3+2 q^2) s_{(3,2,1)}^y \Bigr)$\\
\hline
\end{tabular}
$$


\end{document}